\documentclass[11pt]{article}
\usepackage{amsfonts,amsthm,amsmath,url}
\usepackage{authblk}

\newtheorem{theorem}{Theorem}[section]

\newtheorem{example}{Example}[section]

\newtheorem{lemma}[theorem]{Lemma}

\renewcommand{\SS}{\ensuremath{\mathcal{S}}}
\newcommand{\K}{\ensuremath{\mathcal{K}}}
\newcommand{\eff}{\ensuremath{\mathbb{F}}}

\title {Combinatorial Repairability for Threshold Schemes\thanks{The authors' research is supported by  NSERC discovery grants.}}
\author[1]{Douglas~R.~Stinson}
\author[2]{Ruizhong Wei}
\affil[1]{David R.\ Cheriton School of Computer Science, University of Waterloo,
Waterloo, Ontario, N2L 3G1, Canada}
\affil[2]{Department of Computer Science, Lakehead University, Thunder Bay, Ontario, P7B 5E1, Canada}
\date{\today}
\begin{document}
\maketitle

\begin{abstract}
In this paper, we consider methods whereby a subset of players in a $(k,n)$-threshold 
scheme can ``repair'' another player's share in the event that their share has been lost or corrupted.
This will take place without the participation of the dealer who set up the scheme.
The repairing protocol should not compromise the (unconditional) security of the threshold scheme, and
it should be efficient, where efficiency is measured in terms of  the amount of information 
exchanged during the repairing process. We study two approaches to repairing. The first method is
based on the ``enrollment protocol'' from \cite{NSG} which was originally developed to add a new
player to a threshold scheme (without the participation of the dealer) after the scheme was set up.
The second method distributes ``multiple shares'' to each player, as defined by a suitable
combinatorial 
design. This method results in larger shares, but lower communication complexity, as compared to the first method.
\end{abstract}

\section{Introduction}

Suppose that $k_1, k_2$ and $n$ are positive integers
such that  $k_1 < k_2 \leq n$. 
Informally, a \emph{$(k_1,k_2,n)$-ramp scheme} is 
a method whereby a {\it dealer} chooses a {\it secret} and distributes a {\it share} to 
each of $n$ {\it players} such that the following two properties
are satisfied:
\begin{description}
\item[reconstruction] Any subset of $k_2$ players can compute the
secret from the shares that they collectively hold, and 
\item[secrecy] No subset of $k_1$ players can determine any information about 
the secret.
\end{description}
{We call $k_1$ and $k_2$ the {\it lower threshold} and {\it upper threshold}
of the scheme, respectively.}
When $k_2 = k_1 + 1 = k$, a ramp scheme is known as a \emph{$(k,n)$-threshold scheme}.

In this paper, we are only interested in schemes that are {\it unconditionally secure}.
That is, all security results are valid against adversaries with unlimited computational power.

The original motivation for ramp schemes (as opposed to threshold schemes)
is that ramp schemes permit larger secrets be shared
for a given share size. The efficiency of secret sharing is often measured in terms of the
\emph{information rate} of the scheme, which is defined to be the ratio
$\rho = \log _2 | \K | / \log _2 |  \SS |$
(where $\SS$ is the set of all possible shares and $\K$ is the set of all
possible secrets). That is, the information rate is the ratio of the size of the secret to the
size of a share.

For a threshold scheme, a fundamental result
states that $\rho \leq 1$. 
However, there are constructions 
for ramp schemes
where the (optimal) information
rate is $k_2 - k_1$; for non-threshold ramp schemes, this quantity exceeds one.

We briefly describe a standard construction for ramp schemes with optimal
information rate (see, e.g., \cite{OK}). In the threshold case, this is just the classical
Shamir threshold scheme \cite{Shamir}. The construction takes place over a finite field $\eff_Q$, where $Q \geq n+1$.
\begin{enumerate}
\item In the {\bf Initialization Phase},  
the dealer, denoted by $D$, chooses $n$ distinct, 
non-zero elements of $\eff_Q$, denoted $x_i$, 
$1 \leq i \leq n $. The values $x_i$ are public.
For $1 \leq i \leq n$, $D$ gives the value $x_i$ to $P_i$.  
\item Let $\lambda = k_2 - k_1$. In the {\bf Share Distribution} phase, 
$D$ chooses a secret \[K = (a_0, \dots, a_{\lambda-1}) \in (\eff_Q)^{\lambda}.\]
Then $D$ secretly chooses (independently and uniformly at random)
 $a_{\lambda}, \dots , a_{k_2-1} \in \eff_Q$.
Next, for $1 \leq i \leq n$, 
$D$ computes $y_i= a(x_i)$, where 
\[a(x) = \sum _{j=0}^{k_2-1} a_j\, x^j.\]
Finally, for $1 \leq i \leq n$, $D$
constructs the share $y_i= a(x_i)$ and gives it to $P_i$.
\end{enumerate}

Reconstruction is easily accomplished using the Lagrange interpolation formula
(see, e.g., \cite[\S13.1]{CTAP}).

\subsection{Share Repairability}

The problem of {\it share repairability} has been considered by 
several authors in recent years (see, for example, \cite{GLF}). We will mainly consider repairability
of threshold schemes. The problem setting is that a 
certain player $P_{\ell}$ (in a $(k,n)$-threshold scheme, say) loses their share. The goal is 
to find a ``secure'' protocol involving $P_{\ell}$ and a subset of the other
players that allows the missing share $y_{\ell}$ to be reconstructed.
(Of course the dealer could simply re-send the share to $P_{\ell}$, 
but we are considering a setting where the dealer is
no longer present in the scheme after the initial setup.)
In general, we will assume secure pairwise channels linking
pairs of players. 

We consider protocols that operate in two phases:

\begin{enumerate}
\item  In the {\bf message exchange phase}, a certain subset of $d$ players (not including $P_{\ell}$)
exchange messages among themselves. The integer $d$ is called the {\it repairing degree.}
We will only consider protocols where each player sends at most one message to any other player,
and every message is sent at the same time.

\item In the {\bf repairing phase}, these same $d$ players each send a message to $P_{\ell}$.
The messages received by $P_{\ell}$ allow $P_{\ell}$'s share to be
reconstructed. Some of the protocols we study only require a repairing phase.
\end{enumerate}

We note that $d \geq k$ is an obvious necessary condition for the existence of such a 
scheme. This is seen as follows. Suppose $k-1$ players could repair another player's
share. Then these $k-1$ players would have $k$ shares, which would enable them
to reconstruct the secret. This is of course not allowed in a $(k,n)$-threshold scheme.

We have to consider what it means for a protocol of this type to be ``secure''.
Our definition of security is motivated by the required threshold property. 
In general, we will consider a coalition of $k - 1$ players. This coalition
may or may not include $P_{\ell}$. We assume that all players execute the protocol
correctly, but the coalition is trying to obtain some information about the secret.
(Thus we are assuming that the coalition is ``honest-but-curious''.)
After executing the protocol, the coalition combines all the information it 
holds. This includes their shares, as well as all messages that they
send or receive during the protocol. All of this information should still 
yield no information about the secret.  If a $(k,n)$-threshold scheme
has a repairability protocol that satisfies
this security requirement, then we say that it is
a \emph{$(k,n,d)$-repairable threshold scheme}, which we  abbreviate
to \emph{$(k,n,d)$-RTS}.

We distinguish between two types of repairability in this paper.
We will say that an $(n,k,d)$-RTS has \emph{universal repairability} if {\it any} subset 
 of $d$ players can repair a share of any other player. Most 
previous discussions of repairability in the literature have implicitly or explicitly
considered this model. A weaker condition would be to require only that {\it there 
exists} a subset of $d$ players who will be able to repair a given share belonging
to some other player. We will call this  \emph{restricted repairability}.
 
One potential advantage of considering restricted repairability is that
it can lead to more efficient schemes, where efficiency is measured in
terms of information rate (of the threshold scheme) and/or communication complexity
(of the repairing process). This is one of the themes
we explore in this paper.

\subsection{Our Contributions}

We present two repairability schemes in this paper. The first scheme is
a modification of an enrollment protocol due to Nojoumian {\it et al}.\ described in \cite{NSG,Noj}. 
In this scheme, any $k$ users 
are able to repair a share of another user, and the scheme provides
universal repairability. Thus it is a $(k,n,k)$-RTS. 
The underlying threshold scheme is just the Shamir secret sharing scheme,
which is an ideal scheme (i.e., the information rate is equal to 1).

The second scheme provides restricted repairability. 
It combines two schemes and can lead to a solution with
higher information rate and lower communication complexity
(so it trades off larger share sizes for less information communicated during repairing). It 
uses a distribution design having certain properties to allocate 
subsets of shares of a Shamir scheme (or a ramp scheme) to each user.
We look at various types of combinatorial designs that yield
good solutions for repairability when used in this way.

The rest of the paper is organized as follows.
In Section \ref{enrol.sec},
we present the enrollment protocol, modified to provide repairability.
 In Section \ref{GLF.sec}, we give a brief overview of the 
 Guang-Lu-Fu Scheme \cite{GLF}. Section \ref{combin.sec} presents our second scheme,
 which has a somewhat similar flavour. Then, in Section \ref{dist.sec}, we 
 examine various types of distribution designs and the repairable threshold 
 schemes that can be obtained form them. In Section \ref{comp.sec}.
 we compare our construction to the GLF scheme from \cite{GLF}.
 Section \ref{univ.sec} addresses the problem of universal repairability in the combinatorial setting.
 Finally, Section \ref{concl.sec} is a brief conclusion.

\section{NSG Enrollment Protocol}
\label{enrol.sec}

The {\it enrollment protocol} from \cite{Noj,NSG} was introduced to 
create a share for a new player in a threshold scheme, without requiring
the participation of the dealer who initially set up the scheme.
It was also described in a setting where threshold of the scheme 
was to be altered. Here, we discuss a straightforward modification
where the protocol is used to repair a share, without changing
the threshold. This protocol has repairing degree $k$ and achieves universal repairability.

Suppose we have a $({k},n)$-Shamir threshold scheme defined over $\eff_Q$, and we wish to 
repair the share for a player $P_{\ell}$. We assume that this
share is being repaired by players $P_1, \dots , P_{k}$ and ${\ell} > {k}$.
Suppose the share for $P_{\ell}$ is $\varphi_{\ell} = f({\ell})$, where
$f(x) \in \eff_Q[x]$ is a random polynomial of degree at most ${k}-1$ whose
constant term is the secret. The share $\varphi_{\ell}$ can be 
expressed as
\begin{equation}
\label{eq0}
\varphi_{\ell} = \sum_{i=1}^{k} \gamma_i \varphi_i,
\end{equation}
where the $\gamma_i$'s are public Lagrange coefficients
(see, e.g., \cite[\S13.1]{CTAP}).
In what follows, all arithmetic is performed in $\eff_Q$.

The enrollment protocol  proceeds as follows:

\begin{enumerate}
\item For all $1 \leq i \leq {k}$, player $P_i$
computes random values $\delta_{j,i}$ for $1 \leq j \leq {k}$
such that 
\begin{equation}
\label{eq1}
\gamma_i \varphi_i  = \sum_{j=1}^{k} \delta_{j,i}.
\end{equation}

\item For all $1 \leq i \leq {k}$,  $1 \leq j \leq {k}$, 
player $P_i$ transmits $\delta_{j,i}$ to player $P_j$
using a secure channel.

\item For all $1 \leq j \leq {k}$, player $P_j$ computes
\begin{equation}
\label{eq2}
\sigma_j  = \sum_{i=1}^{k} \delta_{j,i}.
\end{equation}

\item For all $1 \leq j \leq {k}$, player  $P_j$ transmits $\sigma_{j}$ to player $P_{\ell}$
using a secure channel.

\item Player $P_{\ell}$ computes their share $\varphi_{\ell}$
using the formula
\begin{equation}
\label{eq3}
\varphi_{\ell} = \sum_{j=1}^{k} \sigma_j.
\end{equation}

\end{enumerate}

It is straightforward to verify that player $P_{\ell}$ computes their share correctly, i.e.,
the value of $\varphi_{\ell}$ computed using (\ref{eq1}), (\ref{eq2}) and (\ref{eq3}) is the
same as (\ref{eq0}).

Let us consider the security of this protocol. We assume that all players act honestly
during the protocol and do not reveal any information while the protocol is being executed.
Later, however, it may be the case that a coalition $\mathcal{C}$ of ${k}-1$ participants 
attempts to compute some information
about the secret. We will show that this is impossible. Note that we are  basically
describing  the security
proof from \cite[\S 2.4.2c]{Noj} with a few additional details added.

First, we note that computing the secret, given ${k}-1$ shares, is equivalent to 
computing any additional share. This is easy to see, because any ${k}$ shares 
allow the secret to be computed, and any ${k}-1$ shares along with the secret allow
any other share to be computed (this is a well-known property of the Shamir scheme).

There are two cases to consider:
\begin{description}
\item [case (i)] The coalition $\mathcal{C}$  consists of a subset of ${k}-1$ players
in $\{P_1, \dots , P_{k}\}$.
\item [case (ii)] The coalition $\mathcal{C}$  consists of $P_{\ell}$ along with a subset of ${k}-2$ players
in $\{P_1, \dots , P_{k}\}$.
\end{description}

It is convenient to consider the following \emph{share-exchange matrix} defined in \cite{Noj}:
\[ \mathcal{E} = \left( 
\begin{array}{cccc}
\delta_{1,1} & \delta_{2,1} & \cdots & \delta_{{k},1} \\
\delta_{1,2} & \delta_{2,2} & \cdots & \delta_{{k},2} \\
\vdots & \vdots & \ddots & \vdots \\
\delta_{1,{k}} & \delta_{2,{k}} & \cdots & \delta_{{k},{k}} .
\end{array}
\right) .\]
Observe from (\ref{eq1}) that the sum of the entries in the $i$th row of $\mathcal{E}$ 
is equal to $\gamma_i \varphi_i$.
Also, from (\ref{eq2}), the sum of the entries in the $j$th column of $\mathcal{E}$ 
is equal to $\sigma_j$, so 
$P_{\ell}$ knows all ${k}$ column sums.
Finally, it is immediate from (\ref{eq1}), (\ref{eq2}) and (\ref{eq3}) 
that the sum of all the entries in $\mathcal{E}$ is equal to $\varphi_{\ell}$.

In case (i), we can assume without loss of generality that $\mathcal{C} = \{P_1, \dots , P_{{k}-1}\}$.
Here the coalition $\mathcal{C}$ possesses all the entries in $\mathcal{E}$ except for 
 $\delta_{{k},{k}}$. But this value is completely random, and knowing this value is equivalent to
 knowing the value of $\varphi_k$, $\varphi_{\ell}$ or the secret. We conclude that $\mathcal{C}$ 
 has no information about the secret in this case.
 
 Case (ii) is a bit more complicated. Here, we can assume without loss of generality 
 that $\mathcal{C} = \{P_1, \dots , P_{{k}-2}, P_{\ell}\}$. The coalition $\mathcal{C}$ possesses all the entries in $\mathcal{E}$ except for the four entries $\delta_{{k}-1,{k}-1}$, $\delta_{{k}-1,{k}}$, 
 $\delta_{{k},{k}-1}$ and $\delta_{{k},{k}}$. 
 Further, since $P_{\ell}$ knows the column sums, the equations 
 \begin{equation}
 \label{S1}
 \sigma_{k-1} = \delta_{{k}-1,{k}-1} + \delta_{{k}-1,{k}}
 \end{equation} and
 \begin{equation}
 \label{S2}\sigma_{k} = \delta_{{k},{k}-1} + \delta_{{k},{k}}
 \end{equation} are known. 
 So we have  two linear equations in four unknowns.
 
 Of course $P_{\ell}$ also knows the value of the share $\varphi_{\ell}$, 
 and $\varphi_{\ell}$ is a known linear combination
 of the ${k}$ shares $\varphi_1, \dots , \varphi_{k}$, as given by (\ref{eq0}). 
 But only the first ${k}-2$ of these shares 
 are known to $\mathcal{C}$.
 
 It is possible to choose arbitrary values for $\delta_{{k}-1,{k}-1}$ and $\delta_{{k},{k}-1}$.
 Thus \[ \varphi_{k-1} = \frac{\delta_{{k}-1,{k}-1} + \delta_{{k},{k-1}}}{\gamma_{k-1}}\] can 
 take on any arbitrary value. Then the values of $\delta_{{k-1},{k}}$ and $\delta_{{k},{k}}$
 (and hence $\varphi_{k}$)
 will be  determined by (\ref{S1}) and (\ref{S2}).

Similarly, we could choose an arbitrary value for $\varphi_{k}$ and then $\varphi_{k-1}$ would be 
determined. In either case, the coalition knows the values of $k-1$ shares, but they have no
information about the individual shares $\varphi_{k-1}$ and $\varphi_{k}$.
 Since this represents all the information available to $\mathcal{C}$, we 
 conclude that $\mathcal{C}$ also
 has no information about the secret in case (ii).

\subsection{Communication Complexity of the Enrollment Protocol}

The {\it communication complexity} of a share repairing scheme is the 
sum of the sizes (i.e., the bit-lengths) of all the messages transmitted during the protocol
divided by the bit-length of the secret. In the enrollment protocol, 
every message is an element of $\eff_Q$, as is the secret. 
Therefore, 
the communication complexity is equal to the total number of messages transmitted. 
It is computed as follows:
\begin{itemize}
\item $k(k-1)$ in step 2,
\item $k$ in step 4, and
\item therefore the total is  $k^2$.
\end{itemize}

\subsection{Ramp Scheme Repairability}

The same protocol works in the case of a ramp scheme. Here we need $k_2$ players
to reconstruct a lost secret. The same Lagrange formula applies in this situation,
since a share is just an evaluation of the polynomial at a particular point.
The security proof  needs to be modified to consider security against coalitions of
$k_1$ players. As was the situation in analyzing the threshold 
scheme, there are two cases to consider:
\begin{description}
\item [case (i)] The coalition $\mathcal{C}$  consists of a subset of ${k_1}$ players
in $\{P_1, \dots , P_{k_2}\}$.
\item [case (ii)] The coalition $\mathcal{C}$  consists of $P_{\ell}$ along with a subset of ${k_1}-1$ players
in $\{P_1, \dots , P_{k_2}\}$.
\end{description}
We briefly outline the proof in the two cases.

In case (i), we can assume without loss of generality that $\mathcal{C} = \{P_1, \dots , P_{{k_1}}\}$.
The coalition $\mathcal{C}$ possesses all the entries in the share-exchange matrix $\mathcal{E}$ except for 
the $\lambda$ by $\lambda$ lower right submatrix of $\mathcal{E}$
(where $\lambda = k_2 - k_1$). 
The entries of this submatrix can be filled in such that they are consistent with 
any possible values of the $\lambda$ shares
$\varphi_{k_1+1}, \dots ,  \varphi_{k_2}$. Therefore, the secret is completely undetermined.

In case (ii), we assume  that $\mathcal{C} = \{P_1, \dots , P_{{k_1-1}}, P_{\ell}\}$.
Then $\mathcal{C}$ possesses all the entries in  $\mathcal{E}$ except for 
the $\lambda +1$ by $\lambda+1$ lower right submatrix of $\mathcal{E}$.
The coalition also knows the value of $\varphi_{\ell}$ as well as the column sums 
$\sigma_{k_1}, \dots , \sigma_{k_2}$. 
Any $\lambda$ rows of this submatrix can be filled in with arbitrary values,
which means that the $\lambda$ corresponding shares can take on arbitrary values.
The values in the remaining row of the submatrix are then determined by the known
column sums, which means that the share corresponding to this row is determined.
So the information available to the coalition consists of $k_2$ known shares, and
it is consistent with any possible values of any $\lambda$ additional shares.
So the coalition has no information about the secret.

In conclusion, we have shown that $\mathcal{C}$ 
has no information about the secret in either of the two cases.

\section{Guang-Lu-Fu (GLF) Scheme}
\label{GLF.sec}

The GLF scheme, described in \cite{GLF}, has a lower information rate
than the enrollment scheme, but also lower
communication complexity. As such, it achieves a tradeoff between these two measures. 
The GLF scheme provides universal repairability and it is
based on linearized polynomials and minimum
bandwidth regeneration (MBR) codes \cite{DGWWR}. 
We do not discuss the scheme in detail, but we will refer to
its basic properties where it is relevant to do so.

We recall one example from \cite{GLF} to illustrate the basic idea.
Example 2 from \cite{GLF} is
a $(2,4)$-threshold scheme with information rate $1/3$.
The secret is an element over $\mathbb{F}_Q$ and each share is a triple
over $\mathbb{F}_Q$. The repairing degree $d=3$. 
Repairing a player works as follows.
Each of three players send one message to the fourth player, where a message is
an element of $\mathbb{F}_Q$. The three messages received enable the three
components of the share to be reconstructed. 
For this scheme, we would say that the total communication
complexity is 3. This is an improvement over the communication
complexity (which is equal to 4) using the enrollment scheme for a $(2,4)$-threshold
scheme. 

\section{A New Technique for Combinatorial Repairability}
\label{combin.sec}

In this section, we present a $(k,n)$-threshold scheme
with low information rate and communication complexity that achieves
restricted repairability. We base our construction on an old technique,
namely giving each player a subset of shares from an underlying threshold
scheme\footnote{This technique has most commonly been considered in the past 
in connection with the construction of secret sharing schemes
for non-threshold access structures; see, e.g., \cite[Theorem 1]{BL}.}.
We will start with an $(\ell,m)$-threshold scheme, say a Shamir scheme,
implemented over a finite field $\eff_Q$. This is called the 
{\it base scheme}. We then give each 
player a certain subset of $d$ of the $m$ shares. A {\it design} consisting
of $n$ blocks of size $d$, defined on a set of $m$ points, will
be used to do this. This design is termed the {\it distribution design.}
The repairing degree will be equal to $d$.

We will call the shares of the base
$(\ell,m)$-threshold scheme {\it subshares}. Each share in the resulting
$(k,n)$-threshold scheme consists of $d$ subshares.
We need to ensure that the threshold
property is satisfied for the resulting
$(k,n)$-threshold scheme, which we call the {\it expanded scheme}. 
We also need to be able to repair the share of any player in the expanded scheme
by judiciously choosing a certain set of other players, who will then send
appropriate subshares to the player whose share is being repaired.

Let the blocks in the distribution design be denoted $B_1, \dots , B_n$ and let $X$
denote the set of $m$ points. The threshold property will be satisfied
in the expanded scheme 
provided that the following two conditions are satisfied:
\begin{enumerate}
\item the union of any $k$ blocks contains at least $\ell$ points, and
\item the union of any $k-1$ blocks contains at most $\ell -1$ points.
\end{enumerate}
 
We are considering a repairing scheme where certain designated players transmit subshares
to the player whose share is being repaired. 
This technique can be applied provided that every point in the distribution design
occurs in at least two blocks (this is a necessary and sufficient condition for
this kind of repairability to be possible). Therefore, if this property is
satisfied, we say that the distribution design is {\it repairable}. 

Suppose we want to repair the share 
corresponding to a block $B$. For each point $x \in B$, we can find another block
that contains $x$ (because the distribution design is repairable). 
The corresponding player can send the subshare corresponding to $x$
to the player whose node is being repaired. 
The communication complexity of the expanded scheme will be equal to $d$, 
since $d$ elements of $\eff_Q$ are transmitted to repair a share of a secret in $\eff_Q$.

It is not a requirement that the $d$ subshares are obtained from $d$ different 
blocks. For example, it could happen that $d=3$, one block contributes two
subshares, and one block contributes one subshare during the repairing process.
However, we will frequently be considering schemes where we have $d$ blocks, 
each of which contributes one subshare. This is analogous to  the model from \cite{GLF},
where it is   assumed
that each player contributes a constant number $\beta$ of ``elements'' to the player 
whose share is begin repaired (where an ``element'' is a subshare or a certain linear
combination of subshares).

It is quite simple to analyze the security of combinatorial repairability. 
The main point to observe is that the information collectively held by any subset of players (after the 
repairing protocol is completed) consists only of their shares in the expanded scheme.
They did not obtain any information collectively that they did not already possess
before the execution of the repairing protocol.
So it is immediate that a set $k-1$ players cannot compute the secret after the
repairing of a share occurs.

\subsection{Using Ramp Schemes as Base Schemes}

We have one additional useful modification to describe.
Suppose that the distribution design satisfies the following two properties.
\begin{enumerate}
\item the union of any $k$ blocks contains at least $\ell_2$ points, and
\item the union of any $k-1$ blocks contains at most $\ell_1$ points,
\end{enumerate}
where $\ell_2 - \ell_1 \geq 1$. In this case we say that the distribution design
is a {\it $(k,\ell_1,\ell_2)$-distribution design}. See Table \ref{tab1} for
a summary of the parameters and required properties of a distribution design.

\begin{table}
\caption{Parameters and properties of a $(k,\ell_1,\ell_2)$-distribution design}
\label{tab1}
\begin{center}
\begin{tabular}{|l|l|}
\hline
$m$ & number of points in the design\\
$n$ & number of blocks in the design ($=$ the number of players)\\
$d$ & block size ($=$ the repairing degree)\\
$k$ & threshold\\
$\ell_1$ & maximum number of points in the union of $k-1$ blocks\\
$\ell_2$ & minimum number of points in the union of $k$ blocks\\ \hline
\end{tabular}
\end{center}
\end{table}

Given a $(k,\ell_1,\ell_2)$-distribution design, we let the base scheme be an
$(\ell_1,\ell_2,m)$-ramp scheme\footnote{Note that, if $\ell_2 - \ell_1 = 1$, then the ramp scheme is
a threshold scheme, and we have the construction  described  in the
previous section.} defined over $\eff_Q$ (this can be done if
$Q \geq m+1$). 
 Then we use the distribution design to distribute
shares to the $n$ players.
This yields a $(k,n)$-threshold scheme (the expanded scheme) having information 
rate $(\ell_2 - \ell_1)/d$.

Repairing works exactly as before, and $d$ subshares, each of which is an element of 
$\eff_Q$, are transmitted to repair a share. However, the secret is now an
element in $(\eff_Q)^{\ell_2 - \ell_1}$, so the 
communication complexity is now $d/ (\ell_2 - \ell_1)$.
(Note that this is just the reciprocal of the information rate of the expanded scheme.) 

\begin{theorem}
\label{thm1}
Suppose there exists a repairable $(k,\ell_1,\ell_2)$-distribution design on $m$ points, having
$n$ blocks of size $d$, and suppose that  $Q \geq m+1$. Then there is a
$(k,n,d)$-RTS  with restricted repairability, having information rate $(\ell_2 - \ell_1)/d$ and
communication complexity $d/ (\ell_2 - \ell_1)$, where every share is in $(\eff_Q)^d$.
\end{theorem}
 
Suppose we have a $(k,\ell_1,\ell_2)$-distribution design on $n$ blocks 
in which every point occurs in at least two blocks, as required in Theorem 
\ref{thm1}. If we take an arbitrary subset of the blocks of such a design,
then it may not be the case that every point occurs in at least two blocks
of the ``smaller'' design. It would be convenient to have a simple method
of selecting subsets of blocks of a design in such a way that this 
property continues to be satisfied. 

Here is the approach we will use to achieve this objective.
We say that a subset of $s$ blocks in a $(k,\ell_1,\ell_2)$-distribution design on $n$ blocks
is a {\it basic repairing set of size $s$} if every point in the design
is contained in at least two blocks in the basic repairing set.
It is obvious that any superset of a basic repairing set 
is repairable. So we have the following result.

\begin{theorem}
\label{thm2}
Suppose there exists a $(k,\ell_1,\ell_2)$-distribution design on $m$ points, having
$b$ blocks of size $d$, and suppose that  $Q \geq m+1$. Suppose that this design
contains a basic repairing set of size $s$. Then, for any $n$ such that 
$s \leq n \leq b$,  there is a  
$(k,n,d)$-RTS   with restricted repairability, having information rate $(\ell_2 - \ell_1)/d$ and
communication complexity $d/ (\ell_2 - \ell_1)$, where every share is an element of $(\eff_Q)^d$.
\end{theorem}

\section{Some Distribution Designs and the Resulting RTS}
\label{dist.sec}

 In this section, we provide some examples of distribution designs and
 describe how they can be used to construct repairable secret sharing schemes.
 The designs we use are Steiner triple systems, resolvable $(m,d,1)$-BIBDs and projective planes. 

\subsection{Steiner Triple Systems}

We first consider using a Steiner triple system as a distribution design.
This only allows certain thresholds, but the number of players 
can take on a large range of values.
A {\it Steiner triple system of order $m$} 
(or, STS$(m)$) has $m$ points and $b= m(m-1)/6$ blocks of
size $3$, and every pair of points occurs in exactly one block.
An STS$(m)$ can also be defined as an $(m,3,1)$-BIBD 
(balanced incomplete 
block design).
For a comprehensive  reference on Steiner triple systems, see \cite{CR}.

Using the blocks of an STS$(m)$ 
as a distribution design would yield repairing degree $d = 3$.
The simplest application would be to take $k=2$.
The union of any two blocks in the design contains at least five points, 
and each block contains
three points. Hence we can take $\ell_1 = 3$, $\ell_2 = 5$ and use 
a $(3,5,m)$-ramp scheme as the base scheme. The expanded scheme will be 
a $(2,n,3)$-RTS having information rate $2/3$ and communication complexity is $3/2$.
This is certainly an  
improved communication complexity as compared to the enrollment protocol
with threshold $k=2$, which has communication complexity $4$. 

We still need to determine the permissible values of $n$ in the above construction.
It will be advantageous to make use of resolvable STS$(m)$. An STS$(m)$ is {\it resolvable}
if the set of $b= m(m-1)/6$ blocks can be partitioned into $b= (m-1)/2$ {\it parallel classes},
where each parallel class consists of $m/3$ disjoint blocks. It is well-known that
a {\it resolvable} STS$(m)$ exists if and only if $m\equiv 3 \pmod{6}$.

Suppose we use a resolvable STS$(m)$ as our distribution design.
Then two parallel classes in this
design comprise a basic repairing set of size  $2m/3$.
As a result, we can accommodate any number $n$ of players such that
$2m/3 \leq n \leq m(m-1)/6$. We have proved the following theorem.

\begin{theorem}
\label{STS.thm}
Suppose $m\equiv 3 \pmod{6}$, $Q$ is a prime power such that $Q \geq m+1$ and 
$2m/3 \leq n \leq m(m-1)/6$. Then there exists a 
$(2,n,3)$-RTS with restricted repairability, with shares from $(\eff_Q)^3$, having information rate $2/3$ and
communication complexity
$3/2$.
\end{theorem}

\begin{example}
The smallest interesting application of Theorem \ref{STS.thm} is when $m = 9$.
The distribution design is a resolvable STS$(9)$, consisting of  four parallel classes of
three blocks. We take two parallel classes to form the basic repairing sent, 
along with an arbitrary subset of the remaining six blocks.
In this way, we can construct a $(2,n,3)$-RTS 
for any $n$ such that $6 \leq n \leq 12$. The scheme has information rate $2/3$ and
communication complexity
$3/2$.  Subshares are elements of $\eff_Q$, where
$Q\geq 11$ is any prime power, and the secret is an element of $(\eff_Q)^2$.
Shares consist of three elements of $\eff_Q$.
\end{example}

\subsubsection{Quadrilateral-free STS}

What if we use an STS to try to construct a scheme with a higher threshold?
The union of two blocks contains at most six points (and equality is achieved if the two blocks are
disjoint). However, it is easy to find sets of three blocks whose union
contains six points (e.g., three blocks of the form $xyz$, $xuv$, $uyw$).
Even four blocks might have a union consisting of six points:
$xyz$, $xuv$, $uyw$, $vzw$. Such a set of four blocks is
known as a {\it quadrilateral} or {\it Pasch configuration}.
However, it is possible to construct Steiner triple systems that do not
contain any Pasch configurations. These designs are termed 
{\it anti-Pasch} Steiner triple systems. An anti-Pasch Steiner triple system
exists for any order $m\equiv 1,3 \pmod{6}$, $m\neq 7,13$ (see \cite{GGW}).

In an anti-Pasch Steiner triple system, the union of two blocks contain 
at most six points, and the union of four blocks contain at least seven points.
Therefore, the expanded scheme is a $(2,4,n$)-ramp scheme. So we have weakened
the desired threshold property in the expanded scheme,
but we still might get something interesting if
we can identify a small repairing set. In fact, infinite classes of resolvable
anti-Pasch Steiner triple systems are known. 
For example, in \cite{CCL}, it is shown that a resolvable
anti-Pasch Steiner triple system of order $m$ exists for any 
positive integer $m \equiv 9 \bmod 18$. We can use any two parallel classes of the 
design as a basic repairing set, as we did in Theorem \ref{STS.thm}.

\subsection{BIBDs with $\lambda = 1$}

Using the blocks of an $(m,4,1)$-BIBD  
as a distribution design would yield a scheme with repair degree $d = 4$.
We have the following result.

\begin{theorem}
\label{m41.thm}
Suppose $m\equiv 4 \pmod{12}$, $Q$ is a prime power such that $Q \geq m+1$ and 
$m/2 \leq n \leq m(m-1)/12$. Then there exists a 
$(2,n,4)$-RTS with restricted repairability, with shares from $(\eff_Q)^4$, 
having information rate $3/4$ and
communication complexity
$4/3$.
\end{theorem}

\begin{proof}
If $m \equiv 4 \pmod{12}$, then there is a resolvable $(m,4,1)$-BIBD.
The union of any two blocks in a $(m,4,1)$-BIBD  contains at least seven points, 
and each block contains
four points. Hence we can take $k=2$, $\ell_1 = 4$ and $\ell_2 = 7$, and use 
a $(4,7,m)$-ramp scheme as the base scheme. The expanded scheme will be 
a $(2,n,4)$-RTS having information rate $3/4$ and repair degree $4$. 
The communication complexity is $4/3$.

Two parallel classes in the
BIBD comprise a basic repairing set of size  $m/2$.
As a result, we can accommodate any value of $n$ such that
$m/2 \leq n \leq m(m-1)/12$.
\end{proof}

As mentioned before, the enrollment protocol
with threshold $k=2$ has communication complexity equal to $4$, so the communication complexity is
lowered quite considerably in Theorem \ref{m41.thm}.

Using the same idea, we can use other known classes of resolvable $(m,d,1)$-BIBDs to construct 
repairable threshold schemes. When $d$ increases, the
threshold may also increase. We illustrate by
stating results for the  cases $d=5$ and $d=8$.
The proofs are similar to Theorem \ref{STS.thm} and \ref{m41.thm}.

\begin{theorem}
\label{m51.thm}
Suppose $m\equiv 5 \pmod{20}$ and there exists a resolvable $(m,5,1)$-BIBD.
Let $Q$ be a prime power such that $Q \geq m+1$ and 
$2m/5 \leq n \leq m(m-1)/20$. Then the following RTS exist:
\begin{enumerate}
\item A 
$(2,n,5)$-RTS with restricted repairability, with shares from $(\eff_Q)^5$, 
having information rate $4/5$ and
communication complexity
$5/4$.
\item A 
$(3,n,5)$-RTS with restricted repairability, with shares from $(\eff_Q)^5$, 
having information rate $2/5$ and
communication complexity
$5/2$.
\end{enumerate}
\end{theorem}

\begin{proof}
The verifications are straightforward. 
We note that the union of two blocks in the BIBD contains either nine or ten points, 
and the union of three blocks in the design contains at least 12 points.
So we can take $\ell_1 = 5$ and $\ell_2 = 9$ when $k=2$, and
$\ell_1 = 10$ and $\ell_2 = 12$ when $k=3$.
\end{proof} 

The first few values of $m$ for which Theorem \ref{m51.thm} can be applied 
are $m =  25, 65$ and $85$. Actually, resolvable $(m,5,1)$-BIBDs
are known to exist for all $m \equiv 5 \bmod 20$ except $m = 45, 345,465,645$ (see \cite{AGY}).

We state the following similar result without proof.

\begin{theorem}
\label{m81.thm}
Suppose $m\equiv 8 \pmod{56}$ and there exists a resolvable $(m,8,1)$-BIBD.
Let $Q$ be a prime power such that $Q \geq m+1$ and 
$m/4 \leq n \leq m(m-1)/56$. Then the following RTS exist:
\begin{enumerate}
\item A 
$(2,n,8)$-RTS with restricted repairability, with shares from $(\eff_Q)^5$, 
having information rate $7/8$ and
communication complexity
$8/7$.
\item A 
$(3,n,8)$-RTS with restricted repairability, with shares from $(\eff_Q)^5$, 
having information rate $5/8$ and
communication complexity
$8/5$.
\item A 
$(4,n,8)$-RTS with restricted repairability, with shares from $(\eff_Q)^5$, 
having information rate $1/4$ and
communication complexity
$4$.
\end{enumerate}
\end{theorem}

The first few values of $m$ for which Theorem \ref{m81.thm} can be applied 
are $m =  64$ and $120$. Another known result is that resolvable $(m,8,1)$-BIBDs
exist for all $m \equiv 8 \bmod 56$, $m > 24480$ (see \cite{AGY}).

\subsection{Projective Planes}

Finally, we examine the possibility of using finite projective planes as 
distribution designs. A \emph{projective plane of order $q$} is a design 
consisting of $m = q^2+q+1$ points and $q^2+q+1$ blocks (or lines), where each block contains 
exactly $d=q+1$ points and every pair of points occurs in exactly one block.
It follows that every point occurs in exactly $q+1$ blocks and any pair of
blocks intersect in exactly one point. 

For basic results on projective planes, see \cite{Sti}.
It is well-known that a projective plane of order $q$ exists
whenever $q$ is a prime or prime power. In this case, we can let the one-dimensional
subspaces of $(\mathbb{F}_q)^3$ be points and  define the two-dimensional
subspaces of $(\mathbb{F}_q)^3$ to be blocks. The result is a 
projective plane of order $q$ known as PG$(2,q)$.

We will use a certain subset of the blocks of the projective plane
as our distribution design. The permissible values of $n$ will be 
determined by the repairability requirement.

First, we consider the minimum and maximum number of points spanned by a set of
$j$ blocks. These values will determine the parameters of the base scheme.

\begin{lemma}
\label{lemma1}
The union of any $j-1$ blocks in a projective plane of order $q$ contains
at most $q(j-1)+1$ points.
\end{lemma}

\begin{proof}
Denote the $j-1$ blocks by $A_0, \dots , A_{j-2}$. Each $A_i$ ($i \geq 1$) contains a point in 
$A_0$, so \[ \left| \bigcup _{i=0}^{j-2} A_i \right| \leq q+1 + (j-2)q = q(j-1)+1.\]
\end{proof}

\begin{lemma} 
\label{lemma2}
For $j \leq q+1$, the union of any $j$ blocks in a projective plane of order $q$ contains
at least $j(q+1 - (j-1)/2)$ points.
\end{lemma}

\begin{proof}
Denote the $j$ blocks by $A_0, \dots , A_{j-1}$. Each $A_i$ (for $1 \leq i \leq q$) 
contains $q+1- i$ points that are not in $\cup_{h=0}^{i-1} A_h$.
It follows that 
\[ \left| \bigcup_{i=0}^{j-1} A_i \right| \geq \sum _{i=0}^{j-1} (q+1 -i) = j(q+1) - \frac{j(j-1)}{2}.\]
\end{proof}

For repairability, we  determine the existence of some good basic repairing sets.
In general, a basic repairing set of size $s$ 
is equivalent to the \emph{dual of a $2$-blocking set} on $s$ points.
Blocking sets in projective planes have been studied by several authors
and various bounds on the minimum size of a blocking set are known
(see, e.g., Ball and Blokhuis \cite{BB}).
One simple (and well-known) construction is to  
choose any three noncollinear points $x$, $y$ and $z$ of the projective plane,
and take all the blocks that contain at least one of these points. 
This yields a basic repairing set of size $3q$.

Here is a well-known construction that sometimes yields basic repairing sets of size 
$s < 3q$. Suppose that $q$ is a square of a prime power. Start with two
disjoint Baer subplanes in PG$(2,q)$ and take all the blocks that contain
a line from either of these two subplanes. There are $2(q + \sqrt{q} + 1)$
such blocks, and  every point in PG$(2,q)$ is contained in at least two of these
blocks. So we have a basic repairing set of size $2(q + \sqrt{q} + 1)$ 
in this case, which is
an improvement asymptotically over the previous construction. 
(However, $q=9$ is the first value that actually yields a smaller basic repairing set
than the ``simple'' construction.)

Table \ref{tab-plane} contains some examples of repairable threshold schemes 
using projective planes as distribution designs. 
We consider various values of $q$ and $k$. The values of $\ell_1$ and $\ell_2$ are
obtained from Lemmas \ref{lemma1} and \ref{lemma2}. For every $n$
such that $s \leq n \leq m = q^2+q+1$, there is a 
$(k,n,q+1)$-RTS having information rate $\rho$ and
communication complexity $1/\rho$.

\begin{table}
\caption{$(n,k,d)$-RTS based on projective planes}
\label{tab-plane}
\begin{center}
\begin{tabular}{cc}
$
\begin{array}{|r|r|r|r|r|c|c|}
\hline
q & d & k & \ell_1 & \ell_2 & n & \rho\\ \hline
3 & 4 & 2 & 4 & 7 & 9 \leq n \leq 13 & 3/4\\
  &   & 3 & 7 & 9 &   &    1/2\\ \hline
4 & 5 & 2 & 5 & 9 & 12 \leq n \leq 21 & 4/5\\
  &   & 3 & 9 & 12 &      & 3/5\\
  &   & 4 & 13 & 14 &      & 1/5\\ \hline
5 & 6 & 2 & 6 & 11 & 15 \leq n \leq 31 & 5/6\\
  &   & 3 & 11 & 15 &      & 2/3\\
  &   & 4 & 16 & 18 &      & 1/3\\
 \hline
\end{array}
$
&
\begin{tabular}{l}
$q=$  order of projective plane\\
$d=$  repairing degree\\
$k=$  threshold\\
$n=$  number of players\\
$\ell_1, \ell_2$ are ramp scheme thresholds\\
$\rho= $  information rate of the scheme
\end{tabular}
\end{tabular}
\end{center}
\end{table}

\section{Comparison with the GLF Scheme}
\label{comp.sec}

We are able to obtain substantially improved information rates
as compared with the GLF scheme from \cite{GLF}.
They prove an upper bound on the information rate of the schemes they construct that
have {\it optimal repairing rate}.
Optimal repairing rate means that the information received by the user whose share is being repaired
has the same size as a share. Our combinatorial schemes also have this 
feature, so a direct comparison is relevant. The bound obtained in 
\cite{GLF} has the form
\begin{equation}
\label{GLFbound} \rho \leq \frac{k(2d-k+1)}{2dt},
\end{equation}
where 
$t$ is given by the formula 
\begin{equation}
\label{t-eq}
t = \sum_{i=0}^{k-1} \min \{\alpha, (d-i)\beta \} .
\end{equation}
In (\ref{t-eq}), $\alpha$ denotes the number of elements of $\eff_Q$ 
in a share, and each user sends $\beta$ elements of $\eff_Q$ to a user whose 
node is being repaired. Therefore, in our scheme, we have
$\alpha = d$, $\beta = 1$, and hence, from (\ref{t-eq}), we have 
\begin{equation}
\label{t-eq-2}
t = \sum_{i=0}^{k-1} (d-i) = kd - \frac{k(k-1)}{2} .
\end{equation}
Substituting (\ref{t-eq-2}) into (\ref{GLFbound}),
we obtain 
\begin{equation}
\label{GLFbound2} \rho \leq \frac{k(2d-k+1)}{2d\left(kd - \frac{k(k-1)}{2} \right)}
= \frac{2d-k+1}{2d\left(d - \frac{(k-1)}{2} \right)}.
\end{equation}

We illustrate with a couple of examples.

\begin{example} Suppose $k=2$, $d=3$. Then (\ref{GLFbound2}) results in 
$\rho \leq 1/3$. On the other hand, we are able to achieve $\rho = 2/3$ 
for certain values of $n$.
\end{example}

\begin{example} Suppose $k=3$, $d=4$. Then (\ref{GLFbound2}) results in 
$\rho \leq 1/4$. However, we are able to achieve $\rho = 3/4$ 
in certain situations.
\end{example}

We can also compare the communication complexity of our schemes to the 
GLF scheme \cite{GLF}. It is easy to see that the GLF scheme always 
has communication complexity that is at least $d$. On the other hand,
our schemes, as presented in Theorems \ref{thm1} and \ref{thm2},
always have communication complexity that is at most $d$.
(Of course, we also require a suitable distribution design to exist
in order to apply our results.)

\section{Universal Repairability}
\label{univ.sec}

In this section, we consider possible ways to achieve
universal repairability in the combinatorial setting we have introduced.

\subsection{Dual Hypergraph of a Complete Graph}

The first examples of distribution designs for universal repairability 
that we consider allow 
various thresholds, but the number of
players is constrained. The distribution designs are just the dual hypergraphs 
of  complete graphs.  For a positive integer $n$, let $K_{n}$ denote
the complete graph on $n$ vertices. The points of our distribution design
will be the $\binom{n}{2}$ edges of $K_{n}$. For each vertex $x$ of $K_n$, 
we define a block $B_x = \{ e : x \in e\}$. Thus there are $n$ blocks in the design,
each of size $n-1$. Any two blocks intersect in exactly one point, and every point 
occurs in exactly two blocks. The following lemma is proved by a simple counting argument.

\begin{lemma}
\label{Kn} Suppose $1 \leq j \leq n$. The union of any $j$ blocks in the above-described design has cardinality 
$j(n-1) - \binom{j}{2}$.
\end{lemma}

From Lemma \ref{Kn}, for $2 \leq k \leq n$, it follows that the design is a $(k,\ell_1,\ell_2)$-distribution design
on $\binom{n}{2}$ points, where
\[ \ell_1 = (k-1)(n-1) - \binom{k-1}{2}\]
and \[ \ell_2 = k(n-1) - \binom{k}{2}.\]
The design itself constitutes a basic repairing set since every point occurs in 
exactly two blocks. 

We have the following corollary of Theorem \ref{thm2}.

\begin{theorem}
\label{thm-Kn}
Suppose that $n \geq 3$ and $2 \leq k \leq n$.
Denote $m = \binom{n}{2}$ and suppose that  $Q \geq m+1$. Then,  there is a  
$(k,n,n-1)$-RTS with universal repairability, having information rate $(n-k)/(n-1)$ and
communication complexity $(n-1)/ (n-k)$, where every share is an element of $(\eff_Q)^{n-1}$.
\end{theorem}

 \begin{proof} The only observation we need to make is that universal repairability
 and restricted repairability are equivalent when $d = n-1$, since there is only one possible
 set of $d$ players to consider when repairing a given share.
 \end{proof}
 
\subsection{Universal Repairability and $1$-designs}

Suppose the distribution design is a \emph{$(v,b,r,d)$-1-design}. This means that
we have $v$ points, each of which occurs in $r$ blocks, and $b$ blocks in total,
each of which contains $d$ points. We are going to focus on the repairability
property in this section; we do not concern ourselves with the specific 
thresholds that can be achieved.

\begin{theorem}
\label{universal}
A $(v,b,r,d)$-1-design provides universal repairability if and only if 
$b < r+d$.
\end{theorem}

\begin{proof}
Suppose we have a $(v,b,r,d)$-1-design in which $b \geq r + d$.
Suppose $B$ is a block that we want to repair. Let $x \in B$. There
are $r$ blocks that contain $x$, one of which is $B$. Choose any $d$ blocks
that do  not contain $x$ (this can be done because $b-r \geq d$). Then these
$d$ blocks cannot repair the block $B$, since none of these blocks 
contain $x$. 

Conversely, suppose have a $(v,b,r,d)$-1-design in which $b < r + d$.
Let $B$ be a block and let $B_1, \dots B_d$ be any other $d$ blocks.
Then every point $x \in B$ is contained in at least one of these $d$ blocks.
Thus, \[B \subseteq \bigcup _{i=1}^{d} B_i.\]
It follows that the $d$ given blocks are sufficient to repair $B$
(we do not require that each block contributes one subshare, so it is sufficient 
that $B$ is covered by the union of the $d$ blocks).
\end{proof}

The dual hypergraph of the complete graph $K_n$ (as considered in 
the previous section) is
an $(\binom{n}{2},n,2,n-1)$-1-design. Since $n < 2 + (n-1)$, the universal 
repairability property also follows from Theorem \ref{universal}.

Another class of designs that provide universal repairability
are the complements of Hadamard designs.
These are $(4t+3,2t+2,t+2)$-BIBDs and they exist for all $t$ 
such that a Hadamard matrix of order $4t+4$ exists.
We just need to observe that such a BIBD is a $(4t+3,4t+3,2t+2,2t+2)$-1-design.
Since $4t+3 < (2t+2) + (2t+2)$, Theorem \ref{universal}
guarantees that the repairability property holds.

 \section{Summary and Conclusion}
 \label{concl.sec}
 
 We have presented two methods for repairing secrets in threshold schemes.
 The first method is a simple modification of the enrollment protocol
 and the second method is based on using a suitable combinatorial design
 to distribute ``subshares'' of a threshold or ramp scheme. Our schemes
 provide improved information rates and/or communication complexity 
 as compared to previously known schemes.

\section*{Acknowledgements}

The first author would like to thank Nabiha Asghar and Charlie Colbourn 
for helpful comments.

\end{document}